\documentclass[11pt]{article}
\usepackage[dvips]{graphicx}
\usepackage{amsmath,amssymb,amsthm,mathtools}

\usepackage{tikz}

\DeclareFontEncoding{OT2}{}{}
\DeclareFontSubstitution{OT2}{cmr}{m}{n}
\DeclareFontFamily{OT2}{cmr}{\hyphenchar\font45 }
\DeclareFontShape{OT2}{cmr}{m}{n}{
<5><6><7><8><9>gen*wncyr
<10><10.95><12><14.4><17.28><20.74><24.88>wncyr10}{}
\DeclareFontShape{OT2}{cmr}{b}{n}{
<5><6><7><8><9>gen*wncyb
<10><10.95><12><14.4><17.28><20.74><24.88>wncyb10}{}
\DeclareMathAlphabet{\mathcyr}{OT2}{cmr}{m}{n}
\DeclareMathAlphabet{\mathcyb}{OT2}{cmr}{b}{n}
\SetMathAlphabet{\mathcyr}{bold}{OT2}{cmr}{b}{n}

\numberwithin{equation}{section}

\newtheorem{thm}{Theorem}[section]
\newtheorem{prop}[thm]{Proposition}
\newtheorem{lem}[thm]{Lemma}
\newtheorem{cor}[thm]{Corollary}

\theoremstyle{definition}

\numberwithin{equation}{section}

\newcommand{\N}{\mathbb{N}}
\newcommand{\Q}{\mathbb{Q}}
\newcommand{\R}{\mathbb{R}}

\newcommand{\bk}{\mathbf{k}}
\newcommand{\bl}{\mathbf{l}}

\newcommand{\RR}{\mathcal {R}}

\newcommand{\cast}{\circledast}
\newcommand{\bast}{\mathbin{\bar{*}}}
\newcommand{\z}{\zeta}
\newcommand{\rev}[1]{\overleftarrow{#1}}
\newcommand{\ones}[1]{{\underbrace{1,\ldots,1}_{#1}}}
\newcommand{\subones}[1]{{\underbrace{\scriptstyle 1,\ldots,1}_{#1}}}

\DeclareMathOperator{\Li}{\mathrm{Li}}

\def\={\,=\,}

\font\fivecy=wncyr5  \def\sa{{\hbox{\fivecy X}}}

\newcommand{\bvertex}[1]{\fill[black] (#1) circle[radius=2pt]}
\newcommand{\wvertex}[1]{\fill[white] (#1) circle[radius=2pt]; \draw (#1) circle[radius=2pt]}

\title{A generalized regularization theorem and Kawashima's relation for multiple zeta values}
\author{Masanobu Kaneko, Ce Xu, and Shuji Yamamoto}

\begin{document}
\maketitle

\renewcommand{\thefootnote}{\fnsymbol{footnote}}
\footnote[0]{This work was supported by JSPS KAKENHI JP16H06336, JP18K03221, JP18H05233. }
\renewcommand{\thefootnote}{\arabic{footnote}}

\vspace{-20pt}
\begin{abstract}
Kawashima's relation is conjecturally one of the
largest classes of relations among multiple zeta values.
Gaku Kawashima introduced and 
studied a certain Newton series, which we call the Kawashima
function, and deduced his relation by establishing several
properties of this function. 
  We present  a new approach to 
the Kawashima function without using Newton series.
We first establish a generalization of the theory of regularizations of divergent multiple zeta values 
to Hurwitz type multiple zeta values, and then relate 
it to the Kawashima function. Via this connection, 
we can prove a key property of the Kawashima function 
to obtain Kawashima's relation.\end{abstract}

\section{Introduction}

For integers $k_1,\ldots,k_r$ with $k_r>1$ and a parameter $x$ (a real number with $|x|<1$), let us consider 
the multiple zeta value of Hurwitz type 
\begin{equation}\label{HMZV}
\zeta^{(x)}(k_1,\ldots,k_r):=\sum_{0<m_1<\cdots<m_r}\frac1{(m_1+x)^{k_1}\cdots (m_r+x)^{k_r}}.
\end{equation}
In this paper, we first discuss in \S\ref{sec:reg} two regularizations, the series and the integral regularization, of $\zeta^{(x)}(k_1,\ldots,k_r)$ when $k_r=1$, 
and establish a connection between the two, generalizing the well-known regularization theorem for classical 
multiple zeta values (see for instance \cite{IKZ}). 

As an application, we shall give in \S\ref{sec:Kawashima} a different approach to Kawashima's relation of multiple zeta values.
Instead of using the Newton series (\cite{Kaw}),
we define the Kawashima function $F(\bk;x)$ for each index $\bk$ by a Taylor series (\eqref{Kawashima} in \S\ref{sec:Kawashima})
and establish a formula expressing $F(\bk;x)$ in terms of stuffle-regularized polynomials (Theorem~\ref{Kawbyreg}).
This enables us to deduce easily the stuffle product property \eqref{Kawstuf} of $F(\bk;x)$ and thereby 
obtain Kawashima's relation (Corollary~\ref{Kawrel}).

\section{Regularizations}\label{sec:reg}
In this section, we briefly review the theory of regularization of multiple zeta values, and 
extend the theory to Hurwitz type multiple zeta values \eqref{HMZV}.

For any index set $\bk=(k_1,\ldots,k_r)\in\N^r$, we consider two quantities
\[
\z_N^{(x)}(\bk)=\zeta_N^{(x)}(k_1,\ldots,k_r)
\coloneqq \sum_{0<m_1<\cdots<m_r<N}\frac1{(m_1+x)^{k_1}\cdots (m_r+x)^{k_r}}
\]
and 
\[
\Li_\bk^{(x)}(t)=\Li_{k_1,\ldots,k_r}^{(x)}(t)
\coloneqq \sum_{0<m_1<\cdots<m_r}\frac{t^{m_r+x}}{(m_1+x)^{k_1}\cdots (m_r+x)^{k_r}},
\]
where $N$ is a fixed positive integer and $t$ is another real parameter with $0<t<1$. 
If $r=0$, in which case we denote $\bk=\varnothing$, 
these are understood as $\zeta_N^{(x)}(\varnothing)=\Li_\varnothing^{(x)}(t)=1$. 
These quantities are finite, but diverge if $k_r=1$ as $N\to\infty$ or $t\to1$ respectively. 
The following proposition says that there exist polynomials describing the divergence.

\begin{prop}\label{regpoly}
For each index set $\bk$, there uniquely exist polynomials $Z_\ast^{(x)}(\bk;T)$ and $P^{(x)}(\bk;T)$
in $\R[T]$ characterized by
\[\zeta_N^{(x)}(\bk)=Z_\ast^{(x)}\bigl(\bk;\log N-\psi(1+x)\bigr)
+O\bigl(N^{-1}\log^p N\bigr)\quad\text{as}\ N\to\infty\quad (\exists\, p>0) \]
and
\[\Li_\bk^{(x)}(t)=P^{(x)}\bigl(\bk;-\log(1-t)\bigr)
+O\bigl((1-t)\log^p(1-t)\bigr)\quad\text{as}\ t\to1\quad (\exists\, p>0). \]
Here, $\psi(1+x)$ is the digamma function $\Gamma'(1+x)/\Gamma(1+x)$. 
\end{prop}
 
\begin{proof}  
The uniqueness is clear if exist, because a polynomial in $\log N-\psi(1+x)$ (resp.\ in $-\log(1-t)$) 
which converges to 0 when $N\to\infty$ (resp. $t\to 1$) is identically zero.

As for the existence of $Z_\ast^{(x)}(\bk;T)$, we may proceed algebraically as in the classical case $x=0$. 
First, when $\bk=\varnothing$, we have $Z_\ast^{(x)}(\varnothing;T)=P^{(x)}(\varnothing;T)=1$ by convention.
Noting that
$\zeta_N^{(x)}(\bk)$ obeys the standard stuffle product rule 
$\zeta_N^{(x)}(\bk)\zeta_N^{(x)}(\bl)=\zeta_N^{(x)}(\bk\ast\bl)$
(we extend by linearity the symbol  $\zeta_N^{(x)}$ to the formal sum of indices),
we may write $\zeta_N^{(x)}(\bk)$ as a polynomial in $\zeta_N^{(x)}(1)$ with coefficients which are $\Q$-linear combinations
of $\zeta_N^{(x)}(\bl)$'s with admissible $\bl$ (an index set is called admissible if the last entry is greater than 1). 
For the precise definition of the stuffle (sometimes called harmonic) product $\bk\ast\bl$, see for instance \cite{KY}.
Then the polynomial $Z_\ast^{(x)}(\bk;T)$ is obtained by replacing  $\zeta_N^{(x)}(1)$
by $T$ and the coefficients by the same linear combinations of $\zeta^{(x)}(\bl)$ instead of $\zeta_N^{(x)}(\bl)$. Here we note
the asymptotic behavior of $\z_N^{(x)}(1)$, 
\[ \sum_{n=1}^{N-1}\frac1{n+x}=\log N-\psi(1+x)+O(N^{-1})\quad(N\to\infty), \]
which is easily obtained from the well-known 
\[ \psi(1+x)=-\gamma+\sum_{n=1}^\infty\left(\frac1n-\frac1{n+x}\right)\quad (\gamma: \text{ Euler's constant})\] 
and
\[\sum_{n=1}^{N-1}\frac1n=\log N+\gamma+O(N^{-1}), \quad 
\sum_{n=N}^{\infty}\biggl(\frac{1}{n}-\frac{1}{n+x}\biggr)=O(N^{-1}) \quad(N\to\infty). \]
Also we use the (easy) estimate 
\[ \zeta_N^{(x)}(\bl)=\zeta^{(x)}(\bl)+O\left(N^{-1}\log^p N\right)\quad\text{as}\ N\to\infty\quad (\exists\, p>0) \]
for admissible $\bl$. 

As for the existence of  $P^{(x)}(\bk;T)$, the algebraic method using the standard (integral) shuffle product appears not applicable
because $ \Li_\bk^{(x)}(t)$ does not satisfy the shuffle product rule when $x\ne0$. 
However, we may obtain a concrete form of $P^{(x)}(\bk;T)$
by using the iterated integral expression of $ \Li_\bk^{(x)}(t)$:
\begin{align}
\Li_\bk^{(x)}(t)&=\underset{0<u_1<\cdots<u_k<t}{\int}\frac{u_1^xdu_1}{1-u_1}
\underbrace{\frac{du_2}{u_2}\cdots \frac{du_{k_1}}{u_{k_1}}}_{k_1-1}
\frac{du_{k_1+1}}{1-u_{k_1+1}}\underbrace{\frac{du_{k_1+2}}{u_{k_1+2}}\cdots \frac{du_{k_1+k_2}}{u_{k_1+k_2}}}_{k_2-1}\cdots\notag\\
&\qquad\qquad\qquad\qquad\cdots
\frac{du_{k-k_r}}{u_{k-k_r}}\frac{du_{k-k_r+1}}{1-u_{k-k_r+1}}\underbrace{\frac{du_{k-k_r+2}}{u_{k-k_r+2}}\cdots \frac{du_{k}}{u_{k}}}_{k_r-1}, \label{Liiter}
\end{align}
where $k=k_1+\cdots+k_r$. We postpone the actual computation to the proof of Proposition~\ref{reggenfun}, formula \eqref{Pexpl}
(of course there is no circular argument). 
\end{proof}

Recall the  $\R$-linear map $\rho$ from $\R[T]$  to itself defined via the identity
\begin{equation}\label{rho}
\rho(e^{Ty})=A(y)e^{Ty} 
\end{equation}
in the formal power series algebra $\R[T][[y]]$ on which $\rho$ acts 
coefficientwise, where 
\begin{equation}\label{gamma}
A(y)=e^{\gamma y}\,\Gamma(1+y)=\exp\Biggl(\sum_{n=2}^\infty
\frac{(-1)^n}{n}\zeta(n)y^n\Biggr)\in\R[[y]]. 
\end{equation} 
This map $\rho$ plays a 
crucial role in the theory of regularization of multiple zeta values (see \cite{IKZ}).
We extend the fundamental theorem (Theorem~1 in \cite{IKZ}) to our setting.

\begin{thm}\label{regfundHur}   For any index set $\bk$, we have
\[ P^{(x)}(\bk;T)=\rho\bigl(Z_\ast^{(x)}(\bk;T-\gamma-\psi(1+x))\bigr). \]
\end{thm}

When $x=0$, the polynomial $Z_\ast^{(x)}(\bk;T-\gamma-\psi(1+x))$ becomes the usual stuffle
regularized polynomial $Z_\ast(\bk;T)$ (recall $\psi(1)=-\gamma$), and $P^{(0)}(\bk;T)$
equals the shuffle regularized polynomial $Z_\sa(\bk;T)$. Thus the theorem is a generalization
of the fundamental theorem $Z_\sa(\bk;T)=\rho\left(Z_\ast(\bk;T)\right)$ of regularization 
in the theory of usual multiple zeta values 
(see \cite{IKZ,KY}). 

We prove Theorem~\ref{regfundHur} by establishing the generating function identities as follows.
We treat the case $\bk=(1,\ldots,1)$ separately. 
For non-empty index sets $\bk=(k_1,\ldots,k_r)$ and $\bl=(l_1,\ldots,l_s)$, 
we write $\bk_+=(k_1,\ldots,k_{r-1},k_r+1)$ and 
\[\bk\circledast\bl=\bigl((k_1,\ldots,k_{r-1})*(l_1,\ldots,l_{s-1}),k_r+l_s\bigr). \]
Moreover, $\bl^\star$ denotes the formal sum of $2^{s-1}$ index sets of the form 
$(l_1\square\cdots\square l_s)$, where `$+$' or `\,,\,' is inserted in each $\square$. 
Then we have 
\[\zeta^{(x)}(\bk\cast\bl^\star)=
\sum_{0<m_1<\cdots<m_r=n_s\ge\cdots\ge n_1>0}
\frac{1}{(m_1+x)^{k_1}\cdots(m_r+x)^{k_r}(n_1+x)^{l_1}\cdots(n_s+x)^{l_s}}. \]
 
\begin{prop}\label{reggenfun}  
1)  For $Z_\ast^{(x)}$, we have the identities  
\begin{equation}\label{Zpol1}
 \sum_{m=0}^\infty Z_\ast^{(x)}(\underbrace{1,\ldots,1}_m;T) y^m=
\frac{\Gamma(1+x)e^{\psi(1+x)y}}{\Gamma(1+x+y)}\cdot e^{Ty}
\end{equation}
and
\begin{equation}\label{Zpol}
 \sum_{m=0}^\infty  Z_\ast^{(x)}(\bk_+,\underbrace{1,\ldots,1}_m;T) y^m=
\sum_{m=0}^\infty (-1)^m \z^{(x)}(\bk\cast(\underbrace{1,\ldots,1}_{m+1})^\star) y^m\cdot
\frac{\Gamma(1+x)e^{\psi(1+x)y}}{\Gamma(1+x+y)}\cdot e^{Ty}
\end{equation}
for non-empty index set $\bk$. 

2) For $P^{(x)}$, we have
\begin{equation}\label{Ppol1}
\sum_{m=0}^\infty  P^{(x)}(\underbrace{1,\ldots,1}_m;T) y^m=
\frac{\Gamma(1+x)\Gamma(1+y)}{\Gamma(1+x+y)}\cdot e^{Ty}
\end{equation}
and 
\begin{equation}\label{Ppol}
\sum_{m=0}^\infty  P^{(x)}(\bk_+,\underbrace{1,\ldots,1}_m;T) y^m=
\sum_{m=0}^\infty (-1)^m \z^{(x)}(\bk\cast(\underbrace{1,\ldots,1}_{m+1})^\star) y^m\cdot
\frac{\Gamma(1+x)\Gamma(1+y)}{\Gamma(1+x+y)}\cdot e^{Ty}
\end{equation}
for non-empty $\bk$.
\end{prop}

\begin{proof}
1)  To show \eqref{Zpol1}, recall the identity (Lemma~5.1 in \cite{KY})
\[ \sum_{m=0}^\infty [\underbrace{1,\ldots,1}_m]y^m=\exp_\ast\left(\sum_{n=1}^\infty \frac{(-1)^{n-1}}n [n]y^n\right) \]
in $\RR_\ast[[y]]$, where $\RR_\ast$ is the space of formal $\Q$-linear combinations of indices equipped with the stuffle product $\ast$
(which is commutative and associative).
Applying the $\z_\ast^{(x)}$-regularization on both sides, we obtain
\[ \sum_{m=0}^\infty Z_\ast^{(x)}(\underbrace{1,\ldots,1}_m;T)y^m=
\exp\left(\sum_{n=2}^\infty \frac{(-1)^{n-1}}n \z^{(x)}(n)y^n\right) e^{Ty}.\]
We show that the identity
\[ \exp\left(\sum_{n=2}^\infty \frac{(-1)^{n-1}}n \z^{(x)}(n)y^n\right)=
\frac{\Gamma(1+x)e^{\psi(1+x)y}}{\Gamma(1+x+y)} \]
holds.  The both sides become 1 when $y=0$, 
and the logarithmic derivative with respect to $y$ of the left-hand side is
\[ \sum_{n=2}^\infty (-1)^{n-1}\z^{(x)}(n)y^{n-1}, \]
which is equal to that of the right-hand side because $\z^{(x)}(n)=(-1)^n\psi^{(n-1)}(1+x)/(n-1)!$ 
for $n\ge2$. This proves \eqref{Zpol1}. 

For \eqref{Zpol}, we start with the identity
\[ [\bk_+,\underbrace{1,\ldots,1}_m]=\sum_{i=0}^m (-1)^{m-i} [\bk\cast (\underbrace{1,\ldots,1}_{m+1-i})^\star]
\ast [\underbrace{1,\ldots,1}_i]  \]
in $\RR_\ast$ (the case $\bl=(\underbrace{1,\ldots,1}_m)$ of $(A_\ast)$ in Lemma~5.2 in \cite{KY}). From this we obtain by taking the $\z_\ast^{(x)}$-regularization and by forming the generating series
 \[ \sum_{m=0}^\infty  Z_\ast^{(x)}(\bk_+,\underbrace{1,\ldots,1}_m;T) y^m=
\sum_{m=0}^\infty (-1)^m \z^{(x)}(\bk\cast(\underbrace{1,\ldots,1}_{m+1})^\star) y^m\cdot
\sum_{i=0}^\infty Z_\ast^{(x)}(\underbrace{1,\ldots,1}_i;T)y^i. \]
By \eqref{Zpol1}, we obtain \eqref{Zpol}.

2)  To show \eqref{Ppol1}, we first prove a lemma.

\begin{lem} \label{LemLi11}
For any integer $m\ge 1$, it holds
\begin{equation}\label{Li11}
\Li_{\subones{m}}^{(x)}(t)=
\sum_{j=0}^m (-1)^{m-j} \int_0^t\underbrace{\frac{du}{1-u}\cdots\frac{du}{1-u}\frac{1-u^x}{1-u}du}_{m-j}\cdot \frac{(-\log(1-t))^j}{j!}. 
\end{equation}
The integral on the right is an abbreviated form of the iterated integral
\[ \int_0^t\underbrace{\frac{du}{1-u}\cdots\frac{du}{1-u}\frac{1-u^x}{1-u}du}_{m-j}
=\int_{0<u_1<\cdots<u_{m-j}<t}\frac{du_1}{1-u_1}\cdots\frac{du_{m-j-1}}{1-u_{m-j-1}}
\frac{1-u_{m-j}^x}{1-u_{m-j}}du_{m-j}\]
and is regarded as 1 when $j=m$.
\end{lem}
\begin{proof} 
We start from the expression (which follows from \eqref{Liiter})
\[\Li_{\subones{m}}(t)-\Li_{\subones{m}}^{(x)}(t)
=\int_0^t \frac{1-u^x}{1-u}du\underbrace{\frac{du}{1-u}\cdots\frac{du}{1-u}}_{m-1}\\
=I_t\left(\ \begin{tikzpicture}[baseline=20pt]
\draw (0pt,0pt) node(A){}--(12pt,12pt);
\draw[dotted] (12pt,12pt) node(B){}--(36pt,36pt) node(C){}; 
\draw (A) node [above]{\small $\bar x$}; 
\draw (B) to [bend left=40] node [left,above=4pt] {\small $m-1$} (C); 
\bvertex{A}; \bvertex{B}; \bvertex{C};
\end{tikzpicture}\ \right). \]
Here the diagram represents a labeled poset as in \cite{KY,Y}. 
The symbol $\overset{\bar x}{\bullet}$ corresponds to the form $(1-u^x)du/(1-u)$. In \cite{Y}, we used
$\overset{x}{\bullet}$ instead, but we reserve this notation for $u^xdu/(1-u)$ in the later computation
to show \eqref{Ppol}.  
Moreover, notice that we are considering the integral from $0$ to $t$, represented by the symbol $I_t$. 
Then we compute as 
\begin{align*}
I_t\left(\ \begin{tikzpicture}[baseline=20pt]
\draw (0pt,0pt) node(A){}--(12pt,12pt);
\draw[dotted] (12pt,12pt) node(B){}--(36pt,36pt) node(C){}; 
\draw (A) node [above]{\small $\bar x$}; 
\draw (B) to [bend left=40] node [left,above=4pt] {\small $m-1$} (C); 
\bvertex{A}; \bvertex{B}; \bvertex{C};
\end{tikzpicture}\ \right)&=
I_t\left(\ \begin{tikzpicture}[baseline=0pt]
\draw (0,0) node [above]{\small $\bar x$}; 
\bvertex{A}; 
\end{tikzpicture}\ \right)
I_t\left(\ \begin{tikzpicture}[baseline=25pt]
\draw[dotted] (12pt,12pt) node(B){}--(36pt,36pt) node(C){}; 
\draw (B) to [bend left=40] node [left,above=4pt] {\small $m-1$} (C); 
\bvertex{B}; \bvertex{C};
\end{tikzpicture}\ \right)
-I_t\left(\ \begin{tikzpicture}[baseline=25pt]
\draw (0pt,24pt) node(A){}--(12pt,12pt);
\draw[dotted] (12pt,12pt) node(B){}--(36pt,36pt) node(C){}; 
\draw (A) node [above]{\small $\bar x$}; 
\draw (B) to [bend left=40] node [left,above=4pt] {\small $m-1$} (C); 
\bvertex{A}; \bvertex{B}; \bvertex{C};
\end{tikzpicture}\ \right)\\
&=\cdots\\
&=\sum_{j=0}^{m-1} (-1)^{m-1-j} 
I_t\left(\ \begin{tikzpicture}[baseline=14pt]
\draw[dotted] (0pt,24pt) node(A){}--(24pt,0pt) node(B){};
\draw (A) node [above]{\small $\bar x$}; 
\draw (A) to [bend left=40] node [right=5pt,above] {\small $m-j$} (B); 
\bvertex{A}; \bvertex{B}; 
\end{tikzpicture}\ \right)
I_t\left(\ \begin{tikzpicture}[baseline=25pt]
\draw[dotted] (12pt,12pt) node(B){}--(36pt,36pt) node(C){}; 
\draw (B) to [bend left=40] node [left,above=4pt] {\small $j$} (C); 
\bvertex{B}; \bvertex{C};
\end{tikzpicture}\ \right), 
\end{align*}
which gives \eqref{Li11} since 
\[I_t\left(\ \begin{tikzpicture}[baseline=16pt]
\draw (0pt,0pt) node(C){}--(8pt,8pt) node(B){};
\draw[dotted] (8pt,8pt)--(30pt,30pt) node(D){};
\draw (C) to [bend left=40] node [pos=0.55,above=0.6ex] {$j$} (D); 
\bvertex{C}; \bvertex{B};\bvertex{D};
\end{tikzpicture}\ \right)
=\Li_{\subones{j}}(t)=\frac{(-\log(1-t))^j}{j!}. \qedhere\]
\end{proof}

For any $l>0$, the value of the integral 
\[\int_0^t\underbrace{\frac{du}{1-u}\cdots\frac{du}{1-u}\frac{1-u^x}{1-u}du}_{l}\]
converges as $t\to 1$, and the remainder is estimated as 
\begin{align*}
&\int_0^1\underbrace{\frac{du}{1-u}\cdots\frac{du}{1-u}\frac{1-u^x}{1-u}du}_{l}
-\int_0^t\underbrace{\frac{du}{1-u}\cdots\frac{du}{1-u}\frac{1-u^x}{1-u}du}_{l}\\
&=\int_t^1\Li_{\subones{l-1}}(u)\frac{1-u^x}{1-u}du
=O\bigl((1-t)\log^{l-1}(1-t)\bigr)\quad \text{ as } t\to 1. 
\end{align*}
Indeed, $(1-u^x)/(1-u)$ is bounded by a constant and the estimate 
\begin{equation}\label{eq:estimate of integral}
\int_t^1\Li_{\subones{l-1}}(u)du=O\bigl((1-t)\log^{l-1}(1-t)\bigr)
\end{equation}
is shown by induction on $l$. 
Moreover, we have 
\begin{align*}
\int_0^1\underbrace{\frac{du}{1-u}\cdots\frac{du}{1-u}\frac{1-u^x}{1-u}du}_{l}
&=\sum_{n=1}^\infty (-1)^{n-1}x^n
\int_{0<u_1<\cdots<u_l<1}
\frac{du_1}{1-u_1}\cdots\frac{du_l}{1-u_l}\frac{1}{n!}\biggl(\int_{u_l}^1\frac{dv}{v}\biggr)^n\\
&=\sum_{n=1}^\infty (-1)^{n-1}x^n
\int_0^1\underbrace{\frac{du}{1-u}\cdots\frac{du}{1-u}}_l
\underbrace{\frac{dv}{v}\cdots\frac{dv}{v}}_n\\
&=\sum_{n=1}^\infty (-1)^{n-1}\zeta(\ones{l-1},n+1)\,x^n,
\end{align*}
and therefore we obtain
\begin{equation}\label{P111}
 P^{(x)}(\underbrace{1,\ldots,1}_m;T)
=\sum_{j=0}^m (-1)^{m-j} \biggl(\sum_{n=1}^\infty (-1)^{n-1}\zeta(\ones{m-j-1},n+1)\,x^n\biggr)\, \frac{T^j}{j!} 
\end{equation}
from Lemma~\ref{LemLi11} and the definition of $P^{(x)}(\bk;T)$ (Proposition~\ref{regpoly}).

The generating series of 
\[\sum_{n=1}^\infty  (-1)^{n-1}\zeta(\ones{l-1},n+1) \,x^n\]
is given by the well-known formula
\[1+\sum_{l=1}^\infty (-1)^l \sum_{n=1}^\infty
(-1)^{n-1}\zeta(\ones{l-1},n+1)x^n\,y^l=\frac{\Gamma(1+x)\Gamma(1+y)}{\Gamma(1+x+y)},\]
and thus \eqref{Ppol1} is obtained.

Now we define 
\begin{equation}\label{eq:onesKawashima}
F(\ones{l};x)\coloneqq\begin{cases}
\sum_{n=1}^\infty  (-1)^{n-1}\zeta(\ones{l-1},n+1) \,x^n &(l>0),\\
1 & (l=0), \end{cases}
\end{equation}
hence the above formula and the equation~\eqref{P111} can be written, respectively, as
\begin{equation}\label{eq:AomotoDrinfeld}
\sum_{l=0}^\infty(-1)^l F(\ones{l};x)\,y^l=\frac{\Gamma(1+x)\Gamma(1+y)}{\Gamma(1+x+y)} 
\end{equation}
and 
\begin{equation}\label{P11byF11}
P^{(x)}(\ones{m};T)=\sum_{j=0}^m (-1)^{m-j} F(\ones{m-j};x)\, \frac{T^j}{j!}. 
\end{equation}

Note that, by the duality relation $\zeta(\ones{l-1},n+1)=\zeta(\ones{n-1},l+1)$, 
the function \eqref{eq:onesKawashima} coincides with the Kawashima function \eqref{Kawashima} 
for the index set $\bk=(\ones{l})$. 
The identity 
\begin{equation}\label{F11} 
\int_0^1\underbrace{\frac{du}{1-u}\cdots\frac{du}{1-u}\frac{1-u^x}{1-u}du}_{l}
=F(\ones{l};x) 
\end{equation}
obtained above is a special case of \cite[Proposition~2.8]{Y}. 
However, we should keep in mind that we adopt a different definition 
for the Kawashima function.

We now proceed to the proof of \eqref{Ppol}. 
Again we use the poset diagram for the iterated integral as in the proof of Lemma~\ref{LemLi11}.  
As remarked before, the symbol $\overset{x}{\bullet}$ represents the form $u^xdu/(1-u)$. 
Recalling the integral expression~\eqref{Liiter},
we compute 
\begin{align*}
\Li_{(\bk_+,\subones{m})}^{(x)}(t)&=
 I_t\left(\ \begin{tikzpicture}[baseline=26pt]
\node (K) at (15pt,15pt) {$\bk$}; 
\draw (K) circle [radius=11pt];
\draw (-2pt,-2pt) node(A){}--(K)--(31pt,31pt) node(B){}--(40pt,40pt) node(C){};
\draw[dotted] (40pt,40pt)--(52pt,52pt) node(D){};
\draw (C) to [bend left=40] node [pos=0.55,above=0.6ex] {$m$} (D); 
\draw (A) node[above]{$x$};
\bvertex{A}; \wvertex{B}; \bvertex{C}; \bvertex{D};
\end{tikzpicture}\ \right)\\
&= I_t\left(\ \begin{tikzpicture}[baseline=13pt]
\node (K) at (17pt,17pt) {$\bk$}; 
\draw (K) circle [radius=11pt];
\draw (0pt,0pt) node(A){}--(K)--(33pt,33pt) node(B){};
\draw (A) node[above]{$x$};
\bvertex{A}; \wvertex{B};
\end{tikzpicture}\ \right)
 I_t\left(\ \begin{tikzpicture}[baseline=15pt]
\draw (0pt,0pt) node(C){}--(8pt,8pt) node(B){};
\draw[dotted] (8pt,8pt)--(30pt,30pt) node(D){};
\draw (C) to [bend left=40] node [pos=0.55,above=0.6ex] {$m$} (D); 
\bvertex{C}; \bvertex{B};\bvertex{D};
\end{tikzpicture}\ \right)
- I_t\left(\ \begin{tikzpicture}[baseline=20pt]
\node (K) at (15pt,15pt) {$\bk$}; 
\draw (K) circle [radius=11pt];
\draw (-2pt,-2pt) node(A){}--(K)--(31pt,31pt) node(B){}--(40pt,22pt) node(C){};
\draw[dotted] (40pt,22pt)--(56pt,38pt) node(D){};
\draw (C) to [bend left=40] node [pos=0.55,above=0.6ex] {$m$} (D); 
\draw (A) node[above]{$x$};
\bvertex{A}; \wvertex{B}; \bvertex{C}; \bvertex{D};
\end{tikzpicture}\ \right)\\
&=\cdots\\
&=\sum_{j=0}^m (-1)^{m-j}  I_t\left(\ \begin{tikzpicture}[baseline=15pt]
\node (K) at (17pt,17pt) {$\bk$}; 
\draw (K) circle [radius=11pt];
\draw (0,0) node(A){}--(K)--(34pt,34pt) node(B){}--(43pt,22pt) node(L){};
\draw[dotted] (42pt,22pt)--(57pt,2pt) node(M){};
\draw (L) to [bend left=40] node [pos=0.2,right] {$m-j$} (M); 
\draw (A) node[above]{$x$};
\bvertex{A}; \wvertex{B}; \bvertex{L}; \bvertex{M};
\end{tikzpicture}\ \right)
 I_t\left(\ \begin{tikzpicture}[baseline=16pt]
\draw (0pt,0pt) node(C){}--(8pt,8pt) node(B){};
\draw[dotted] (8pt,8pt)--(30pt,30pt) node(D){};
\draw (C) to [bend left=40] node [pos=0.55,above=0.6ex] {$j$} (D); 
\bvertex{C}; \bvertex{B};\bvertex{D};
\end{tikzpicture}\ \right).
\end{align*}
We derive from this that
\begin{equation}\label{Pexpl}
 P^{(x)}(\bk_+,\underbrace{1,\ldots,1}_m;T) =\sum_{j=0}^m (-1)^{m-j}
 I\left(\ \begin{tikzpicture}[baseline=15pt]
\node (K) at (17pt,17pt) {$\bk$}; 
\draw (K) circle [radius=11pt];
\draw (0,0) node(A){}--(K)--(34pt,34pt) node(B){}--(43pt,22pt) node(L){};
\draw[dotted] (42pt,22pt)--(57pt,2pt) node(M){};
\draw (L) to [bend left=40] node [pos=0.2,right] {$m-j$} (M); 
\draw (A) node[above]{$x$};
\bvertex{A}; \wvertex{B}; \bvertex{L}; \bvertex{M};
\end{tikzpicture}\ \right)\frac{T^j}{j!}, 
\end{equation}
because the estimate
\begin{align*}
&I\left(\ \begin{tikzpicture}[baseline=15pt]
\node (K) at (17pt,17pt) {$\bk$}; 
\draw (K) circle [radius=11pt];
\draw (0,0) node(A){}--(K)--(34pt,34pt) node(B){}--(43pt,22pt) node(L){};
\draw[dotted] (42pt,22pt)--(57pt,2pt) node(M){};
\draw (L) to [bend left=40] node [pos=0.2,right] {$m-j$} (M); 
\draw (A) node[above]{$x$};
\bvertex{A}; \wvertex{B}; \bvertex{L}; \bvertex{M};
\end{tikzpicture}\ \right)
- I_t\left(\ \begin{tikzpicture}[baseline=15pt]
\node (K) at (17pt,17pt) {$\bk$}; 
\draw (K) circle [radius=11pt];
\draw (0,0) node(A){}--(K)--(34pt,34pt) node(B){}--(43pt,22pt) node(L){};
\draw[dotted] (42pt,22pt)--(57pt,2pt) node(M){};
\draw (L) to [bend left=40] node [pos=0.2,right] {$m-j$} (M); 
\draw (A) node[above]{$x$};
\bvertex{A}; \wvertex{B}; \bvertex{L}; \bvertex{M};
\end{tikzpicture}\ \right) \\
&=\int_t^1\Li_{\bk}^{(x)}(u)\Li_{\subones{m-j}}(u)\frac{du}{u}
=O\bigl((1-t)\log^p(1-t)\bigr) \quad (\text{for some $p>0$})
\end{align*}
follows from \eqref{eq:estimate of integral}.   
The identity \eqref{Ppol} will be a consequence of the following lemma.

\begin{lem}\label{Lem2.5}  For $l\ge0$, we have
\[  I\left(\ \begin{tikzpicture}[baseline=15pt]
\node (K) at (17pt,17pt) {$\bk$}; 
\draw (K) circle [radius=11pt];
\draw (0,0) node(A){}--(K)--(34pt,34pt) node(B){}--(43pt,22pt) node(L){};
\draw[dotted] (42pt,22pt)--(57pt,2pt) node(M){};
\draw (L) to [bend left=40] node [pos=0.1,right=3pt] {$l$} (M); 
\draw (A) node[above]{$x$};
\bvertex{A}; \wvertex{B}; \bvertex{L}; \bvertex{M};
\end{tikzpicture}\ \right)
=\sum_{i=0}^l \z^{(x)}\bigl(\bk\cast (\underbrace{1,\ldots,1}_{l+1-i})^\star\bigr)
F(\underbrace{1,\ldots,1}_i;x). \]
\end{lem}

\begin{proof} 
Let $I^{(x)}(\bk_+;l)$ denote the left-hand side. 
We also set $\bk=(k_1,\ldots,k_r)$. 
Then we have 
\begin{align*}  
&I^{(x)}(\bk_+;l)
=\int\limits_{u_0>u_1>\cdots>u_l}
\Li^{(x)}_\bk(u_0)\frac{du_0}{u_0}\frac{du_1}{1-u_1}\cdots \frac{du_l}{1-u_l}\\
&=\int\limits_{u_0>u_1>\cdots>u_l}\sum_{0<m_1<\cdots<m_r}
\frac{u_0^{m_r+x}}{(m_1+x)^{k_1}\cdots(m_r+x)^{k_r}}
\frac{du_0}{u_0}\frac{du_1}{1-u_1}\cdots \frac{du_l}{1-u_l}\\
&=\int\limits_{u_1>\cdots>u_l}\sum_{0<m_1<\cdots<m_r}
\frac{1-u_1^{m_r+x}}{(m_1+x)^{k_1}\cdots(m_r+x)^{k_r+1}}
\frac{du_1}{1-u_1}\cdots \frac{du_l}{1-u_l}.
\end{align*}
Here each variable $u_i$ is in the interval $(0,1)$ and 
we omit the bounds $0$ and $1$ in the notation of the integral.
Now we split the last integral into two by writing 
\[1-u_1^{m_r+x}=1-u_1^{x}+u_1^{x}-u_1^{m_r+x}=(1-u_1^{x})+u_1^{x}(1-u_1^{m_r}),\]
and thus obtain
\begin{multline*}  
I^{(x)}(\bk_+;l)
=\int\limits_{u_1>\cdots>u_l}\sum_{0<m_1<\cdots<m_r}
\frac{1}{(m_1+x)^{k_1}\cdots(m_r+x)^{k_r+1}}
\frac{1-u_1^{x}}{1-u_1}du_1 \frac{du_2}{1-u_2}\cdots \frac{du_l}{1-u_l}\\
+\int\limits_{u_1>\cdots>u_l}\sum_{0<m_1<\cdots<m_r}
\frac{1}{(m_1+x)^{k_1}\cdots(m_r+x)^{k_r+1}}
\frac{u_1^{x}(1-u_1^{m_r})}{1-u_1}du_1\frac{du_2}{1-u_2}\cdots \frac{du_l}{1-u_l}.
\end{multline*}
The infinite series in the first integral is nothing but $\zeta^{(x)}(\bk_+)=\zeta^{(x)}(\bk\cast(1)^\star)$
which is independent of the integral variables there, and the integral equals
$F(\underbrace{1,\ldots,1}_l;x)$ by \eqref{F11}.  So this term $\zeta^{(x)}(\bk\cast(1)^\star)F(\underbrace{1,\ldots,1}_l;x)$
is the term for $i=l$ on the right of the lemma.

As for the second integral, we integrate with respect to $u_1$ after writing 
$\frac{1-u_1^{m_r}}{1-u_1}=\sum_{n_1=1}^{m_r} u_1^{n_1-1}$ to obtain
\begin{align*} 
&\int\limits_{u_1>\cdots>u_l}\sum_{0<m_1<\cdots<m_r}
\frac{1}{(m_1+x)^{k_1}\cdots(m_r+x)^{k_r+1}}
\sum_{n_1=1}^{m_r} u_1^{x+n_1-1}du_1\frac{du_2}{1-u_2}\cdots \frac{du_l}{1-u_l}\\
&=\int\limits_{u_2>\cdots>u_l}\sum_{0<m_1<\cdots<m_r\ge n_1\ge1}
\frac{1-u_2^{n_1+x}}{(m_1+x)^{k_1}\cdots(m_r+x)^{k_r+1}(n_1+x)}
\frac{du_2}{1-u_2}\cdots \frac{du_l}{1-u_l}\\
&=\int\limits_{u_2>\cdots>u_l}\sum_{0<m_1<\cdots<m_r=n_0\ge n_1\ge1}
\frac{1-u_2^{n_1+x}}{(m_1+x)^{k_1}\cdots(m_r+x)^{k_r}(n_0+x)(n_1+x)}\\
&\hspace{300pt}\cdot\frac{du_2}{1-u_2}\cdots \frac{du_l}{1-u_l}.
\end{align*}
As before, writing $1-u_2^{n_1+x}=(1-u_2^{x})+u_2^{x}(1-u_2^{n_1})$ and splitting the integral into two, 
we get from the first term $\zeta^{(x)}(\bk\cast(1,1)^\star)\cdot F(\underbrace{1,\ldots,1}_{l-1};x)$ 
which is the $i=l-1$ term in the lemma. 
We integrate the second term with respect to $u_2$ after writing 
$\frac{1-u_2^{n_1}}{1-u_2}=\sum_{n_2=1}^{n_1} u_2^{n_2-1}$ and the same procedure continues 
to obtain the desired formula.
The proof of the lemma is complete.
\end{proof}

\begin{proof}[Proof of \eqref{Ppol}]
By the equation~\eqref{Pexpl} and Lemma~\ref{Lem2.5}, we have
\begin{align*}
&\sum_{m=0}^\infty  P^{(x)}(\bk_+,\underbrace{1,\ldots,1}_m;T) y^m\\
&=\sum_{m=0}^\infty \sum_{j=0}^m (-1)^{m-j} \sum_{i=0}^{m-j}
\z^{(x)}\bigl(\bk\cast (\ones{m-j+1-i})^\star\bigr)F(\ones{i};x)\, \frac{T^j}{j!}\,y^m\\
&\hspace{-15pt}\underset{l\coloneqq m-j-i}{=} \sum_{l,i,j=0}^\infty(-1)^{l+i}
\zeta^{(x)}\bigl(\bk\cast (\ones{l+1})^\star\bigr)F(\ones{i};x)\, \frac{T^j}{j!}\,y^{l+i+j}. 
\end{align*}
Thus the equation~\eqref{Ppol} follows from \eqref{eq:AomotoDrinfeld}. 
\end{proof}

\noindent{\it Proof of Theorem~\ref{regfundHur}.}  Replacing  $T$ in \eqref{Zpol} by $T-\gamma-\psi(1+x)$, and applying
the map $\rho$ on both sides, we see that the right-hand side becomes that of \eqref{Ppol}, and we conclude the theorem. \end{proof}

\section{The Kawashima function and Kawashima's relation}\label{sec:Kawashima}

For a non-empty index set $\bk$, we define the Kawashima function $F(\bk;x)$ by 
\begin{equation}\label{Kawashima}
F(\bk;x):=\sum_{m=1}^\infty (-1)^{m-1}\zeta\bigl((\underbrace{1,\ldots,1}_m)
\cast(\bk^\vee)^\star\bigr)x^m.
\end{equation}
Here $\bk^\vee$ denotes the Hoffman dual of $\bk$ (see \cite[Definition 6.3]{KY}). 
Originally, Kawashima \cite{Kaw} defined $F(\bk;x)$ via the Newton series and obtained this expression as a Taylor expansion at $x=0$.
We remark here that, because of the duality $\zeta\bigl((\underbrace{1,\ldots,1}_m)\cast(1)\bigr)
=\zeta(\underbrace{1,\ldots,1}_{m-1},2)=\zeta(m+1)$, we have $F(1;x)=\psi(1+x)+\gamma$. 
Recall that $x$ is a real variable with $|x|<1$ throughout the present paper.  

Our main result in this section is the following. 

\begin{thm}\label{Kawbyreg} For any $\bk=(k_1,\ldots,k_r)\in\N^r$, we have 
\[ F(\bk;x)
=\sum_{j=0}^r (-1)^{r-j} Z_\ast^\star(k_1,\ldots,k_j;T)
Z_\ast^{(x)}(k_{r},\ldots,k_{j+1};T-\gamma-\psi(1+x)). \]
\end{thm}

Here, $Z_\ast^\star(\bk;T)\coloneqq Z_\ast(\bk^\star;T)$ is the stuffle regularized polynomial 
for the multiple zeta-star values (see \cite{KY}). 
Let $\bast$ be the stuffle product for $\z^\star(\bk)$'s, that is, the product so defined that 
we have $\z^\star(\bk)\z^\star(\bl)=\z^\star(\bk\bast\bl)$ (also see \cite{KY} for the definition). 
Then, since $Z_\ast^\star$ and $Z_\ast^{(x)}$ satisfy $\bast$- and $\ast$-product rules respectively, 
we conclude by the standard argument using the Hopf algebra structure on $\RR_\ast$ (see \cite{H}), 
that $F(\bk;x)$ satisfies the $\bast$-product rule
(here again we extend $F(\,-\,;x)$ by linearity to the formal sum of indices)
\begin{equation}\label{Kawstuf} 
F(\bk;x)F(\bl;x)=F(\bk\bast\bl;x), 
\end{equation}
and thereby obtain the following Kawashima's relations by comparing the coefficients 
on both sides of \eqref{Kawstuf}.

\begin{cor}\label{Kawrel}  For any $m\ge1$ and any non-empty index sets $\bk$ and $\bl$, we have
\[ \sum_{\substack{p+q=m\\ p,q\ge 1}}
\zeta\bigl((\ones{p})\cast(\bk^\vee)^\star\bigr)
\zeta\bigl((\ones{q})\cast(\bl^\vee)^\star\bigr)
=-\zeta\bigl((\ones{m})\cast((\bk\bast\bl)^\vee)^\star\bigr)\quad (\forall m\ge1). \]
\end{cor}

The next proposition is the key to our proof of the theorem. To state it, we need a variant $F(\bk;x;t)$ 
of the Kawashima function with an additional parameter $t$ ($0<t<1$) defined by  
\[ F(\bk;x;t)=\sum_{m=1}^\infty (-1)^{m-1}
\hat\zeta\bigl((\ones{m});\bk^\vee;1-t\bigr)x^m, \]
where for $\bk=(k_1,\ldots,k_r)$ and $\bl=(l_1,\ldots,l_s)$ we set
\begin{equation}\label{eq:hatzeta}
\hat{\zeta}(\bk;\bl;t):=
\sum_{0<m_1<\cdots<m_r=n_s\geq\cdots \geq n_1>0}
\frac{t^{n_1}}{m_1^{k_1}\cdots m_r^{k_r}n_1^{l_1}\cdots n_s^{l_s}}.
\end{equation}
In particular, the value $\hat{\zeta}(\bk;\bl;1)$ at $t=1$ is equal to $\zeta(\bk\circledast\bl^\star)$. 
Note that the power of $t$ in the sum on the right of the definition~\eqref{eq:hatzeta} is $n_1$, 
not the outer $n_s$. 
In the following we use the formula for the derivative
\[\frac{d}{dt}\hat{\zeta}(\bk;\bl;t)=\begin{cases}
\frac{1}{t}\hat{\zeta}\bigl(\bk;(l_1-1,l_2,\ldots,l_s);t\bigr) & (l_1>1),\\
\frac{1}{1-t}\Bigl[\zeta\bigl(\bk\cast(l_2,\ldots,l_s)\bigr)
-\hat{\zeta}\bigl(\bk;(l_2,\ldots,l_s);t\bigr)\Bigr] & (l_1=1)
\end{cases}\]
when $l_1+\cdots+l_s\geq 2$, which follows easily from the definition \eqref{eq:hatzeta}.
Note that, when $\bl=(1)$, we have $\hat{\zeta}(\bk;1;t)=\Li_{k_1,\ldots,k_r+1}(t)$
and its derivative is well-known. Using these, we obtain 
\begin{equation}\label{eq:Kaw diff t}
\begin{split}
\frac{\partial}{\partial t}F&(k_r,\ldots,k_1;x;t)\\
&=\begin{cases}
-\frac{1}{1-t}F(k_{r-1},\ldots,k_1;x;t) & (k_r=1),\\
\frac{1}{t}\Bigl[F(k_r-1,k_{r-1},\ldots,k_1;x;t)-F(k_r-1,k_{r-1},\ldots,k_1;x)\Bigr] & (k_r>1), 
\end{cases}
\end{split}
\end{equation}
when $k_1+\cdots+k_r>1$. For the case of weight one, we have 
\begin{align}
\frac{\partial}{\partial t}F(1;x,t)
&=\frac{1}{1-t}\sum_{m=1}^\infty(-1)^m\Li_{\subones{m}}(1-t)x^m
=\frac{1}{1-t}\sum_{m=1}^\infty \frac{\log^m t}{m!}x^m\notag \\
&=\frac{t^x-1}{1-t}.  \label{eq:Kaw1 diff t}
\end{align}

\begin{prop}\label{Th2}
For a non-empty index set $\bk=(k_1,\ldots,k_r)$, we have
\begin{equation}\label{XeCu} 
\Li_{\bk}^{(x)}(t)+(-1)^r F(\rev{\bk};x;t)
=\sum_{j=0}^r (-1)^j F(k_j,\ldots,k_1;x)\Li_{k_{j+1},\ldots,k_r}(t),
\end{equation}
where $\rev{\bk}\coloneqq(k_r,\ldots,k_1)$. 
\end{prop}

\begin{proof} 
We proceed by  induction on the weight $k:=k_1+\cdots +k_r$. When the weight is 1, 
the left-hand side is $\Li_1^{(x)}(t)-F(1;x;t)$ and the right-hand side is $\Li_1(t)-F(1;x)$. 
Both of these have the value $-F(1;x)$ at $t=0$.  
From the differential formula \eqref{eq:Kaw1 diff t} together with
\[ \frac{\partial}{\partial t}\Li_1^{(x)}(t)=\frac{t^x}{1-t}, \]
we see that the derivatives with respect to $t$ of both sides coincide. 
Hence the case $k=1$ is proved.

Assume the weight $k$ is greater than 1 and the formula is valid for lower weights.
If $k_r=1$, then we have
\[\frac{\partial}{\partial t}\left(\text{L.H.S.}\right)
=\frac1{1-t}\Li_{k_1,\ldots,k_{r-1}}^{(x)}(t)+\frac{(-1)^{r+1}}{1-t} F(k_{r-1},\ldots,k_1;x;t)  \]
and
\[ \frac{\partial}{\partial t}\left(\text{R.H.S.}\right)
=\sum_{j=0}^{r-1}(-1)^jF(k_j,\ldots,k_1;x)\cdot\frac1{1-t}\Li_{k_{j+1},\ldots,k_{r-1}}(t). \]
The last expression is equal to the previous one by the induction hypothesis and we are done.
If $k_r>1$, then we have
\begin{align*}
\frac{\partial}{\partial t}\left(\text{L.H.S.}\right)
&=\frac{1}{t} \Li_{k_1,\ldots,k_{r-1},k_r-1}^{(x)}(t)\\
&\quad +\frac{(-1)^r}{t}\Bigl[F(k_r-1,k_{r-1},\ldots,k_1;x;t)-F(k_r-1,k_{r-1},\ldots,k_1;x)\Bigr]
\end{align*}
and
\[\frac{\partial}{\partial t}\left(\text{R.H.S.}\right)
=\sum_{j=0}^{r-1}(-1)^jF(k_j,\ldots,k_1;x)\cdot\frac1{t}\Li_{k_{j+1},\ldots,k_{r-1},k_r-1}(t), \]
which are equal by the induction hypothesis.
\end{proof}

\begin{proof}[Proof of Theorem \ref{Kawbyreg}]
Considering the asymptotic behavior of \eqref{XeCu} as $t\to1$, we obtain
\[ P^{(x)}(\bk;T)=\sum_{j=0}^r (-1)^j F(k_j,\ldots,k_1;x)Z_\sa(k_{j+1},\ldots,k_r;T). \]
Applying the map $\rho^{-1}$ on both sides and using Theorem~\ref{regfundHur}, we obtain
\begin{equation}\label{ZstarbyF}
Z_\ast^{(x)}(\bk;T-\gamma-\psi(x+1))=\sum_{j=0}^r (-1)^j F(k_j,\ldots,k_1;x)Z_\ast(k_{j+1},\ldots,k_r;T). 
\end{equation}
From this, we obtain Theorem~\ref{Kawbyreg} (and vice versa) by using the `andipode relation' 
\[ \sum_{j=0}^r (-1)^{r-j} Z_\ast^\star(k_1,\ldots,k_j;T)Z_\ast(k_r,\ldots,k_{j+1};T)=\begin{cases} 1 & (r=0), \\
0 & (r>0) \end{cases} \]
coming from the Hopf algebra structure on $\RR_\ast$.    
\end{proof}

\section{Various remarks}

\noindent {\bf 1.}
The right-hand side of Theorem~\ref{Kawbyreg} is independent of $T$, so that we may set 
$T=0$ or $T=\gamma+\psi(1+x)$ to obtain
\[ F(\bk;x)
=\sum_{j=0}^r (-1)^{r-j} \zeta_\ast^\star(k_1,\ldots,k_j)
Z_\ast^{(x)}(k_{r},\ldots,k_{j+1};-\gamma-\psi(1+x)) \]
or
\[ F(\bk;x)
=\sum_{j=0}^r (-1)^{r-j} Z_\ast^\star(k_1,\ldots,k_j;\gamma+\psi(1+x))
\zeta_\ast^{(x)}(k_{r},\ldots,k_{j+1}), \]
or from \eqref{ZstarbyF}
\[ \z_\ast^{(x)}(\bk)=\sum_{j=0}^r (-1)^j F(k_j,\ldots,k_1;x)Z_\ast(k_{j+1},\ldots,k_r;\gamma+\psi(x+1)). \]
In particular, if all $k_i$ are larger than 1, no regularization is needed and we have
\[ \zeta^{(x)}(k_1,\ldots,k_r)
=\sum_{j=0}^r (-1)^j F(k_j,\ldots,k_1;x)\zeta(k_{j+1},\ldots,k_r)\]
and equivalently
\[  F(k_1,\ldots,k_r;x)
=\sum_{j=0}^r (-1)^{r-j} \zeta^\star(k_1,\ldots,k_j)\zeta^{(x)}(k_r,\ldots,k_{j+1}).\]

\bigskip
\noindent {\bf 2.}
There is yet another route to reach Theorem~\ref{Kawbyreg}, by adopting a different definition 
of the Kawashima function as a starting point.

For index sets $\bk=(k_1,\ldots,k_r)$, define $F(\bk;x)$ inductively by
\[ F(\varnothing;x)=1,\quad F(\bk;x)=\sum_{n=1}^\infty \left(
\frac{F(\bk';n)}{n^{k_r}}-\frac{F(\bk';n+x)}{(n+x)^{k_r}}\right), \]
where $\bk'=(k_1,\ldots,k_{r-1})$. 
Then we can show that the identity 
\begin{align}
\notag &F(k_1,\ldots,k_r;x)\\
\label{eq:partial fraction}
&=\lim_{N\to\infty}\,
\sum_{j=1}^r (-1)^{r-j}\sum_{\substack{0<n_1\leq\dots\leq n_j<N\\ n_j>n_{j+1}>\dots>n_r>0}}
\frac{1}{n_1^{k_1}\dots n_{j-1}^{k_{j-1}}}
\biggl(\frac{1}{n_j^{k_j}}-\frac{1}{(n_j+x)^{k_j}}\biggr)\\
\notag&\hspace{200pt}\times\frac{1}{(n_{j+1}+x)^{k_{j+1}}\dots (n_r+x)^{k_r}}
\end{align}
holds. For a fixed $N$, we see that the (finite) sum on the right is equal to 
\[ \sum_{j=0}^r (-1)^{r-j} \zeta_N^\star(k_1,\ldots,k_j)\,\zeta_N^{(x)}(k_{r},\ldots,k_{j+1}). \]
Noting the asymptotic behaviors 
\[ \zeta_N^\star(k_1,\ldots,k_j)=Z_\ast^\star(k_1,\ldots,k_j;\log N+\gamma)
+O\left(N^{-1}\log^p N\right) \quad(\exists\,p>0) \]
and
\[ \zeta_N^{(x)}(k_{r},\ldots,k_{j+1})=Z_\ast^{(x)}(k_{r},\ldots,k_{j+1};\log N-\psi(1+x))
+O\left(N^{-1}\log^p N\right) \quad(\exists\,p>0),\]
we conclude that the polynomial 
\[\sum_{j=0}^r (-1)^{r-j} Z_\ast^\star(k_1,\ldots,k_j;T+\gamma)
Z_\ast^{(x)}(k_{r},\ldots,k_{j+1};T-\psi(1+x)) \]
(obtained by replacing $\log N$ by $T$) is independent of $T$, and is equal to $F(k_1,\ldots,k_r;x)$. 
By replacing $T$ by $T-\gamma$ we obtain Theorem~\ref{Kawbyreg}. 

We note that a formula quite analogous to \eqref{eq:partial fraction} 
for a variant of $F(\bk;x)$ (also defined inductively as above) has been
obtained by K.~Ihara and Y.~Nakamura. We surmise that their function has an analogous expression 
in terms of some regularized polynomials as ours.  

\bigskip
\noindent {\bf 3.}  
For any admissible index set $\bk$, the Hurwitz multiple zeta value $\zeta^{(x)}(\bk)$ 
has a nice Taylor expansion at $x=0$:
\[\zeta^{(x)}(\bk)=\sum_{m=0}^\infty \zeta_\sa(\bk^\dagger,\underbrace{1,\ldots,1}_m)\,x^m. \]
Here, $\zeta_\sa$ on the right is the shuffle regularized value and 
$\bk^\dagger$ is the usual dual of $\bk$ (in the notation of our paper, 
$\zeta_\sa(\bk^\dagger,\underbrace{1,\ldots,1}_m)$ equals 
$P^{(0)}(\bk^\dagger,\underbrace{1,\ldots,1}_m;0)$). 
This can be deduced by combining Propositon~3.9 and Theorem~2.5 in \cite{KT}. 
We may also start from the integral expression 
\[\zeta^{(x)}(\bk)=I\left(\ \begin{tikzpicture}[baseline=10pt]
\node (K) at (17pt,17pt) {$\bk$}; 
\draw (K) circle [radius=11pt];
\draw (0,0) node(A){}--(K); 
\draw (A) node[above]{$x$};
\bvertex{A};  
\end{tikzpicture}\ \right), \]
which is the equation~\eqref{Liiter} at $t=1$. Then we expand $u^x$ in the integral as 
\[u^x=\sum_{m=0}^\infty \frac{(\log u)^m}{m!}\,x^m
=\sum_{m=0}^\infty \frac{(-1)^m}{m!}\biggl(\int_u^1\frac{dv}{v}\biggr)^m\,x^m
=\sum_{m=0}^\infty (-1)^m \int\limits_{u<v_1<\cdots<v_m<1}
\frac{dv_1}{v_1}\cdots\frac{dv_m}{v_m}\,x^m\]
to obtain
\[ \zeta^{(x)}(\bk)=\sum_{m=0}^\infty (-1)^m 
I\left( 
\begin{tikzpicture}[baseline=15pt]
\node (K) at (35pt,19pt) [draw,circle]{$\bk$}; 
\draw[dotted] (0,34pt)node(A){}--(14pt,14pt); 
\draw (14pt,14pt)node(B){}--(23pt,2pt)node(C){}--(K); 
\wvertex{A}; \wvertex{B}; \bvertex{C}; 
\draw (A) to [bend right=60] node [midway,below] {{\small $m$}} (B); 
\end{tikzpicture}
\ \right)  \, x^m.\] 
By the duality and the regularization formula \cite[Prop.~8]{IKZ}, we have
\[ I\left( 
\begin{tikzpicture}[baseline=15pt]
\node (K) at (35pt,19pt) [draw,circle]{$\bk$}; 
\draw[dotted] (0,34pt)node(A){}--(14pt,14pt); 
\draw (14pt,14pt)node(B){}--(23pt,2pt)node(C){}--(K); 
\wvertex{A}; \wvertex{B}; \bvertex{C}; 
\draw (A) to [bend right=60] node [midway,below] {{\small $m$}} (B); 
\end{tikzpicture}
\ \right) =  I\left(\ \begin{tikzpicture}[baseline=15pt]
\node (K) at (17pt,17pt) {$\bk^\dagger$}; 
\draw (K) circle [radius=11pt];
\draw (K)--(34pt,34pt) node(B){}--(43pt,22pt) node(L){};
\draw[dotted] (43pt,22pt)--(57pt,2pt) node(M){};
\draw (L) to [bend left=40] node [pos=0.1,right=3pt] {$m$} (M); 
\wvertex{B}; \bvertex{L}; \bvertex{M};
\end{tikzpicture}\ \right)
=(-1)^m \zeta_\sa(\bk^\dagger,\underbrace{1,\ldots,1}_m).
\]
We should also note that, by the ``integral-series identity'' (\cite[Th.~4.1]{KY}), we have
\[ \zeta_\sa(\bk^\dagger,\underbrace{1,\ldots,1}_m)
=(-1)^m\zeta((\bk^\dagger)_-\cast (\underbrace{1,\ldots,1}_{m+1})^\star),\]
where $(\bk^\dagger)_-$ is the index obtained from $\bk^\dagger$ 
by subtracting 1 from the last component.

\end{document}